\documentclass[twoside,12pt,reqno]{amsart}
\usepackage{amssymb,latexsym,verbatim,pstricks} % verbatim for ``comment''
\usepackage{hyperref,url}
\hypersetup{citecolor=red, linkcolor=blue, colorlinks=true}

\theoremstyle{plain}
\newtheorem{theorem}{Theorem}

\newtheorem{corollary}[theorem]{Corollary}

\theoremstyle{definition}

\theoremstyle{definition}              % Bold heading, roman text
\renewcommand{\geq}{\geqslant}
\renewcommand{\leq}{\leqslant}

\begin{document}

\newcommand{\TC}{T_{\textup{\texttt{C}}}}
\newcommand{\C}{\mathcal{C}}
\newcommand{\eps}{\varepsilon}
\newcommand\GL{\textup{GL}}
\newcommand\n{\newline}
\newcommand{\N}{\mathbb{N}}
\newcommand{\p}{\mathfrak{p}}
\newcommand{\SL}{\textup{SL}}
\newcommand{\Q}{\mathbb{Q}}
\newcommand{\R}{\mathbb{R}}
\def\st#1{\textup{\tiny #1}}
\def\sL{\textup{\tiny $L$}}\def\sR{\textup{\tiny $R$}}
\newcommand{\Z}{\mathbb{Z}}

\hyphenation{Frob-enius}
%\hyphenation{http:-www.cwu.edu/$\sim$glasbys/}

\begin{center}\large\sffamily\mdseries
  Distant parents in complete binary trees
\end{center}

\title[\tiny\upshape\rmfamily Distant parents in complete binary trees]{}

%\date{Draft printed on \today}

\author{\sffamily S.\,P. Glasby}

\address[Glasby]{
Department of Mathematics\\   
Central Washington University\\
WA 98926, USA. Also affiliated with The Faculty of Information
Sciences and Engineering, University of Canberra and The Centre for Mathematics of Symmetry and Computation at the
University of Western Australia. {\tt
 GlasbyS@gmail.com, \href{http://www.cwu.edu/~glasbys/}
{http://www.cwu.edu/$\sim$glasbys/}}}.

%\address[]{\kern -4.4mm Department of Mathematics\\   
%Central Washington University\\
%WA 98926-7424, USA.\newline
%\url{http://www.cwu.edu/~glasbys/}}
%\href{http://www.cwu.edu/~glasbys/}
%     {URL: \tt http:/\kern-1pt/www.cwu.edu/$\sim$glasbys/}\qquad
%\href{mailto:glasbys@gmail.com}{Email: \tt glasbys@gmail.com}
%}
%\newline {\tt http:/\kern-1pt/www.cwu.edu/$\sim$glasbys/}}

\begin{abstract}

There is a unique path from the root of a tree to any other vertex.
Every vertex, except the root, has a parent: the adjoining vertex
on this unique path. This is the conventional definition of the parent
vertex. For complete binary trees, however, we show that it is useful
to define another parent vertex, called a {\it distant parent}.
The study of distant parents leads to novel connections with dyadic
rational numbers. Moreover, we apply the concepts of close and distant
parent vertices to deduce an apparently new sense in which continued fractions
are `best' rational approximations.
\end{abstract}

\maketitle

\vskip-2mm
\centerline{\noindent\Small 2010 Mathematics subject classification: 05C05,
20E08, 11A55, 13-01}
\centerline{\noindent\Small Keywords: Parent, children, infinite complete
binary tree, string, inorder, continued fraction}

\section{Introduction}

There is a unique path from the root of a tree to any other vertex.
Hence each vertex in a tree with at most two children vertices can be
associated with a string $S$ of lefts and rights, an {\it $LR$-string}.
We shall focus exclusively on (infinite) complete binary trees, where each
vertex has precisely two children vertices. We henceforth abbreviate
these as `trees.' Rather than {\it identifying} the
the vertices with $LR$-strings, it will be convenient to initially {\it label}
vertices with $LR$-strings, see Fig.~\ref{F:Srev}. The empty $LR$-string
is denoted by $\eps$. Thus the left and right children vertices of $S$
are $C_L(S):=SL$ and $C_R(S):=SR$, respectively. If $S\neq\eps$, then
the conventional parent vertex is obtained by deleting the last symbol~of~$S$.

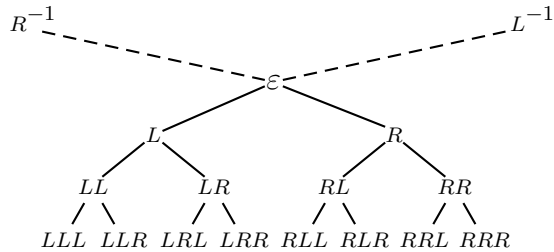
\begin{figure}[!ht]
  \begin{center}
  \psset{xunit=4mm}\psset{yunit=7mm}%\psset{xunit=5mm}\psset{yunit=9mm}
  \begin{pspicture}(0,0)(15,4)
    \thicklines
    \rput(1,0){$\sL\sL\sL$}\rput(3,0){$\sL\sL\sR$}\rput(5,0){$\sL\sR\sL$}
    \rput(7,0){$\sL\sR\sR$}\rput(9,0){$\sR\sL\sL$}\rput(11,0){$\sR\sL\sR$}
    \rput(13,0){$\sR\sR\sL$}\rput(15,0){$\sR\sR\sR$}
    \rput(2,1){$\sL\sL$}\rput(6,1){$\sL\sR$}\rput(10,1){$\sR\sL$}
    \rput(14,1){$\sR\sR$}
    \rput(4,2){$\sL$}\rput(12,2){$\sR$}
    \rput(8,3){$\eps$}\rput(0,4.2){$\sR^{-1}$}\rput(16.6,4.2){$\sL^{-1}$}
    \psline(1.3,0.3)(1.7,0.7)\psline(5.3,0.3)(5.7,0.7)
    \psline(9.3,0.3)(9.7,0.7)\psline(13.3,0.3)(13.7,0.7)
    \psline(2.3,0.7)(2.7,0.3)\psline(6.3,0.7)(6.7,0.3)
    \psline(10.3,0.7)(10.7,0.3)\psline(14.3,0.7)(14.7,0.3)
    \psline(2.3,1.2)(3.7,1.85)\psline(10.3,1.2)(11.7,1.85)
    \psline(5.7,1.2)(4.3,1.85)\psline(13.7,1.2)(12.3,1.85)
    \psline(4.3,2.07)(7.7,2.92)\psline(8.3,2.92)(11.7,2.15)
    \psline[linestyle=dashed](0.3,3.96)(7.7,3.04)
    \psline[linestyle=dashed](8.3,3.04)(15.7,3.96)
  \end{pspicture}
  \end{center}
  \caption{Infinite complete binary tree with vertices labeled
  by $LR$-strings.}\label{F:Srev}
\end{figure}

In this paper we say that every string $S$ has a left and right
parent vertex denoted $P_L(S)=SR^{-1}$ and $P_R(S)=SL^{-1}$, respectively.
The expressions $SR^{-1}$ and $SL^{-1}$ are evaluated recursively using the rules:
$LL^{-1}=\eps$, $RR^{-1}=\eps$, $LR^{-1}=R^{-1}$, and~$RL^{-1}=L^{-1}$\label{p:}.
Thus when $S=L^2R^2$, for example, $P_L(S)=L^2R$ and $P_R(S)=L$.
Every vertex $S\neq\eps$ has a {\it close} and a {\it distant} parent
vertex denoted $P_C(S)$ and $P_D(S)$, respectively. The former is
the usual definition of parent, and the latter is studied in this note.

The aim of this note is to relate close and distant parents to
dyadic\footnote{A dyadic rational is one whose denominator is a power of two.}
rationals via simple recurrence relations, or explicit formulae,
see Theorems~\ref{T:main} and~\ref{T:r} in Section~\ref{S:main}.
Properties of distant parents are described using three metrics:
the length $|S|$ of a string, its position $N(S)$ on a tree, and
an order-preserving linear metric $r(S)$ defined later. Note that
$|S|$ and $N(S)$ are natural numbers, while $r(S)$ is a dyadic rational number
satisfying $0\leq r(S)\leq 2$.

Infinite complete binary trees have strong connections with
group theory~\cite{S06,S80}, with the theory of automata, and with the
analysis of computer programs. However, this largely expository note focuses on
elementary examples. An outline of this paper is as follows.
Section~\ref{S:main} relates close and distant parents of a vertex~$S$
to the the numbers
$|S|$, $N(S)$, and $r(S)$. In Section~\ref{S:app}, an infinite complete
binary tree whose vertices are continued fractions is considered.
The children vertices are most naturally defined in terms
of close and distant parents.
Continued fractions are well-known to be associated with
best rational approximations, see~\cite[4.5.3.\;Ex.\;42]{K2}
and~\cite[p.\;112]{GKP}. For example, if $a=[a_0,a_1,\dots]$
and $b=[b_0,b_1,\dots]$ are irrational numbers, then the
rational number $c$ between $a$ and $b$ with smallest numerator or denominator
is $c=[a_0,\dots,a_{k-1},\min(a_k,b_k)+1]$ where $a_i=b_i$ for $0\leq i<k$
and $a_k\neq b_k$. We shall show in Theorem~\ref{T:app1}(d) that
the close and distant parents to a continued fraction are the best lower-level
rational approximations on a complete binary tree of all rationals.

\section{The main results}\label{S:main}

In this section we define length $|S|$ of a string~$S$, 
its position $N(S)$ on a tree, and an order-preserving linear
function $r(S)$, see Fig.~\ref{F:rN}. These are related to the parent
vertices of $S$. Let $S(k_0,k_1,\dots,k_m)$ denote the $LR$-string
\begin{equation}\label{E:S}
  S(k_0,k_1,\dots,k_m):=\begin{cases}
      R^{k_0}L^{k_1}\cdots L^{k_{m-1}}R^{k_m}&\textup{if $m$ is even;}\\
      R^{k_0}L^{k_1}\cdots R^{k_{m-1}}L^{k_m}&\textup{if $m$ is odd;}\\
      \end{cases}
\end{equation}
where $k_0\in\N:=\{0,1,2,\dots\}$ and $k_i\geq1$ if $1\leq i\leq m$.
The {\it length} of $S=S(k_0,k_1,\dots,k_m)$ is defined to be
$|S|=k_0+k_1+\dots+k_m$. It counts the number of $LR$-symbols in $S$,
and gives the {\it level} of $S$ in the tree shown in Fig.~\ref{F:Srev}.
The {\it position} $N(S)$ of a string $S$ is determined
by Fig.~\ref{F:rN}(b), and a formula for $N(S)$ is given in
Theorem~\ref{T:r} below.

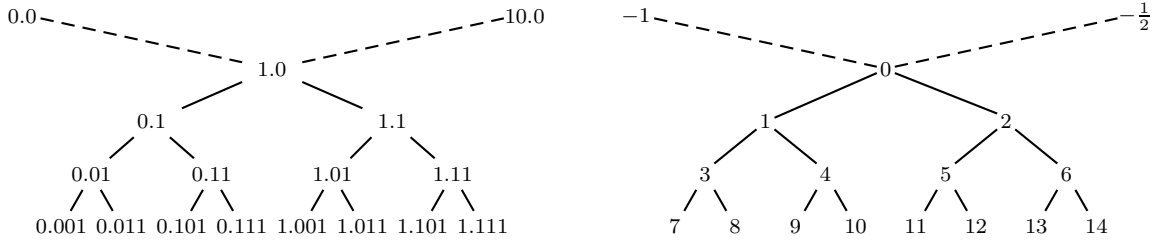
\begin{figure}[!ht]
  \begin{center}
  \psset{xunit=4mm}\psset{yunit=7mm}%\psset{xunit=5mm}\psset{yunit=9mm}
  \begin{pspicture}(-0.5,0)(16.3,4)
    \thicklines
    \rput(1,0){\st{$0.001$}}\rput(3,0){\st{$0.011$}}\rput(5,0){\st{$0.101$}}
    \rput(7,0){\st{$0.111$}}\rput(9,0){\st{$1.001$}}\rput(11,0){\st{$1.011$}}
    \rput(13,0){\st{$1.101$}}\rput(15,0){\st{$1.111$}}
    \rput(2,1){\st{$0.01$}}\rput(6,1){\st{$0.11$}}\rput(10,1){\st{$1.01$}}
    \rput(14,1){\st{$1.11$}}
    \rput(4,2){\st{$0.1$}}\rput(12,2){\st{$1.1$}}
    \rput(8,3){\st{$1.0$}}\rput(-0.3,4){\st{$0.0$}}\rput(16.4,4){\st{$10.0$}}
    \psline(1.3,0.3)(1.7,0.7)\psline(5.3,0.3)(5.7,0.7)
    \psline(9.3,0.3)(9.7,0.7)\psline(13.3,0.3)(13.7,0.7)
    \psline(2.3,0.7)(2.7,0.3)\psline(6.3,0.7)(6.7,0.3)
    \psline(10.3,0.7)(10.7,0.3)\psline(14.3,0.7)(14.7,0.3)
    \psline(2.6,1.3)(3.4,1.7)\psline(10.6,1.3)(11.3,1.7)
    \psline(5.4,1.3)(4.6,1.7)\psline(13.4,1.3)(12.6,1.7)
    \psline(5,2.25)(7,2.75)\psline(9,2.75)(11,2.25)
    \psline[linestyle=dashed](0.3,3.96)(7,3.125)%(7.7,3.04)
    \psline[linestyle=dashed](9,3.125)(15.7,3.96)
%    \psline[linestyle=dashed](1,3.875)(7,3.125)
%    \psline[linestyle=dashed](9,3.125)(15,3.875)
  \end{pspicture}
  \hskip15mm
  \begin{pspicture}(0,0)(16.5,4)
    \thicklines
    \rput(1,0){\st{$7$}}\rput(3,0){\st{$8$}}\rput(5,0){\st{$9$}}
    \rput(7,0){\st{$10$}}\rput(9,0){\st{$11$}}\rput(11,0){\st{$12$}}
    \rput(13,0){\st{$13$}}\rput(15,0){\st{$14$}}
    \rput(2,1){\st{$3$}}\rput(6,1){\st{$4$}}\rput(10,1){\st{$5$}}
    \rput(14,1){\st{$6$}}
    \rput(4,2){\st{$1$}}\rput(12,2){\st{$2$}}
    \rput(8,3){\st{$0$}}\rput(-0.3,4){\st{$-1$}}\rput(16.3,4){\st{$-\frac12$}}
    \psline(1.3,0.3)(1.7,0.7)\psline(5.3,0.3)(5.7,0.7)
    \psline(9.3,0.3)(9.7,0.7)\psline(13.3,0.3)(13.7,0.7)
    \psline(2.3,0.7)(2.7,0.3)\psline(6.3,0.7)(6.7,0.3)
    \psline(10.3,0.7)(10.7,0.3)\psline(14.3,0.7)(14.7,0.3)
    \psline(2.3,1.2)(3.7,1.85)\psline(10.3,1.2)(11.7,1.85)
    \psline(5.7,1.2)(4.3,1.85)\psline(13.7,1.2)(12.3,1.85)
    \psline(4.3,2.07)(7.7,2.92)\psline(8.3,2.92)(11.7,2.15)
    \psline[linestyle=dashed](0.3,3.96)(7.7,3.04)
    \psline[linestyle=dashed](8.3,3.04)(15.7,3.96)
  \end{pspicture}
  \end{center}
  \caption{(a) binary expansions of $r(S)$; and (b) position values $N(S)$.}
    \label{F:rN}
\end{figure}

The vertices of the tree in Fig.~\ref{F:Srev} can be ordered from
left-to-right as the real numbers~$r$ in the interval $0<r<2$ are so ordered.
Consider a vertical line through the vertex (string) $S$ meeting the
horizontal interval $0<r<2$ at the real number~$r(S)$.
The elements of the monoid $\{L,R\}^*:=\{\eps,L,R,L^2,LR,R^2,L^3,\dots\}$
will be called {\it strings}, and those of
$\{L,R\}^\bigstar:=\{L^{-1},R^{-1}\}\cup\{L,R\}^*$
will be called {\it generalized string}.
%Let $\{L,R\}^*$ denote the monoid of all $LR$-strings.
A convenient recursive definition of $r$ is:
\begin{equation}\label{E:rdef}
  r(\eps)=1,\quad r(SL)=r(S)-2^{-|SL|},\quad r(SR)=r(S)+2^{-|SR|}
  \quad\textup{for $S\in\{L,R\}^*$}.
\end{equation}
A simple induction shows that $r(S)$ is a dyadic rational.
The value $r(R^{-1})=0$
is obtained from $r(SR)=r(S)+2^{-|SR|}$ by substituting $S=R^{-1}$.
Similarly, $r(L^{-1})=2$ is obtained from $r(SL)=r(S)-2^{-|SL|}$ by
substituting $S=L^{-1}$. Thus $r$ extends to generalized strings.
(Inorder traversal of a {\it finite} binary tree~\cite[\S2.3.1]{K}
coincides with $r$-ordering.)

Moving left decreases the $r$-value, and moving right increases the $r$-value.
However, a left move is not counteracted by any number of right moves; nor is
a right move counteracted by any number of left moves.
That is, 
\begin{equation}\label{E:LRorder}
  r(SL)<r(SLR)<r(SLR^2)<\cdots<r(S)<\cdots<r(SRL^2)<r(SRL)<r(SR).
\end{equation}
Thus $r\colon\{L,R\}^\bigstar\to[0,2]$ is an injective function which
orders the generalized strings.

The generalized strings $L^{-1}$ and $R^{-1}$ have length (or level) $-1$,
by definition. It follows from Theorem~\ref{T:main}(c) below that
the $S'\in\{L,R\}^\bigstar$ with $|S'|<|S|$ and $r(S')<r(S)$ which {\it maximizes} $r(S')$ is $P_L(S)=S'$.
Similarly, the $S'$ with $|S'|<|S|$ and $r(S')>r(S)$ which {\it minimizes}
$r(S')$ is $P_R(S)=S'$.
For this reason, the parents of a vertex in a tree are commonly
`best approximations' (in some sense) to the vertex.

\begin{theorem}\label{T:main}
Let $S\in\{L,R\}^*$ be a string, and let $m\geq0$ be an integer. Then
\begin{itemize}
  \item[(a)] $\{r(S)\mid\, |S|=m\}$ equals
  $\{\frac{2k-1}{2^m}\mid 1\leq k\leq 2^m\}$;
  \item[(b)] $\{r(S)\mid 0\leq |S|\leq m\}$ equals
  $\{\frac{\ell}{2^m}\mid 1\leq \ell\leq 2^{m+1}-1\}$;
  \item[(c)] the following recurrences hold
  \begin{align*}
    P_L(\eps)&=R^{-1},&&P_L(SL)=P_L(S), &&P_L(SR)=S;\\
    P_R(\eps)&=L^{-1},&&P_R(SL)=S, &&P_R(SR)=P_R(S);
  \end{align*}
  \item[(d)] $\max(|P_L(S)|,|P_R(S)|)<|S|$,
  and $|P_L(S)|\neq|P_R(S)|$ if $S\neq\eps$;
  \item[(e)] $r(P_L(S))=r(S)-2^{-|S|}$ and $r(P_R(S))=r(S)+2^{-|S|}$;
%  \begin{align*}
%    P_L(\eps)&=R^{-1},\qquad P_L(SL)=P_L(S),\qquad P_L(SR)=S;\\
%    P_R(\eps)&=L^{-1},\qquad P_R(SL)=S,\qquad P_R(SR)=P_R(S);
%  \end{align*}
  \item[(f)] if $S=S(k_0,\dots,k_m)$, $X=S(k_0,\dots,k_{m-1},k_m-1)$,
  $Y=S(k_0,\dots,k_{m-2},k_{m-1}-1)$, then the following formulas hold
  \begin{equation}\label{E:ParStr}
    P_L(S)=\begin{cases}
        X&\textup{if $m$ is even;}\\
        Y&\textup{if $m$ is odd;}\\
        \end{cases}\hskip20mm
    P_R(S)=\begin{cases}
        Y&\textup{if $m$ is even;}\\
        X&\textup{if $m$ is odd.}\\
        \end{cases}
  \end{equation}
\end{itemize}
\end{theorem}

\begin{proof}
(a)~We use induction on $m$. The result is true when $m=0$ as $r(\eps)=1$.
Assume now that $|S|=m>0$. Then $S$ equals $S'L$ or $S'R$ where $|S'|=m-1$.
By induction, $r(S')=\frac{2k-1}{2^{m-1}}$ for a unique $k$ with
$1\leq k\leq 2^{m-1}$.
It follows from (\ref{E:rdef}) that $r(S'L)=\frac{4k-3}{2^m}$ and
$S'R=\frac{4k-1}{2^m}$, see Fig.~\ref{F:consec}.
These $r$-values equal $\frac{2\ell-1}{2^m}$ for a unique $\ell$ with
$1\leq\ell\leq 2^m$. This proves part~(a).
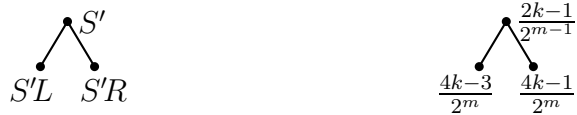
\begin{figure}[!ht]
  \begin{center}
  \psset{xunit=6mm,yunit=6mm}
  \begin{pspicture}(-1,-0.5)(1,1.3)
    \psline(-0.6,0)(0,1)\psline(0.6,0)(0,1)
    \rput(-0.8,-0.5){$S'\kern-2pt L$}\rput(0.8,-0.5){$S'\kern-2pt R$}\rput(0.55,1){$S'$}
    \pscircle*(0,1){0.06}\pscircle*(-0.6,0){0.06}\pscircle*(0.6,0){0.06}
  \end{pspicture}
  \hskip45mm
  \begin{pspicture}(-1,-0.5)(1,1.3)
    \psline(-0.6,0)(0,1)\psline(0.6,0)(0,1)
    \rput(-0.9,-0.6){$\frac{4k-3}{2^m}$}
    \rput(0.9,-0.6){$\frac{4k-1}{2^m}$}\rput(0.9,1){$\frac{2k-1}{2^{m-1}}$}
    \pscircle*(0,1){0.06}\pscircle*(-0.6,0){0.06}\pscircle*(0.6,0){0.06}
  \end{pspicture}
  \end{center}
  \caption{Children rules (a) for strings; (b) for $r$-values.}
  \label{F:consec}
\end{figure}

(b)~By part~(a), $\{r(S)\mid 0\leq |S|\leq m\}$ equals
$\bigcup_{i=0}^m\{\frac{2k-1}{2^i}\mid 1\leq k\leq 2^i\}$. This is
a disjoint union as the fractions are reduced, and hence distinct. The
union has
$\sum_{i=0}^m 2^i=2^{m+1}-1$ fractions, each of which is more than zero,
and less than two.
Since each fraction can be written with a denominator of $2^m$, the union
equals $\{\frac{\ell}{2^m}\mid 1\leq \ell\leq 2^{m+1}-1\}$, as desired.

(c)~The initial conditions and the recurrences follow from the definition of
$P_L(S)$ and $P_R(S)$ together with the rules for postmultiplying by
$L^{-1}$ and $R^{-1}$ on page~\pageref{p:}.

(d)~Since $P_L(S)=SR^{-1}$ and $P_R(S)=SL^{-1}$, and $|S|$ counts the number
of symbols ($L$s and $R$s) in $S$, it follows that
$|P_L(S)|<|S|$ and $|P_R(S)|<|S|$.
Suppose $S\neq\eps$. One of $P_L(S)$ and $P_R(S)$, the close parent,
has length $|S|-1$ because precisely one symbol is canceled.
For the distant parent, however, at least two symbols are canceled (this
needs appropriate interpretation if $S=L$ or $R$). Hence
$|P_L(S)|\neq|P_R(S)|$ if $S\neq\eps$, and the parents of $S$ lie on
different levels.

(e)~We use induction on $|S|$. The result is true for $|S|=0$.
Suppose now that $|S|>0$. Then $S=S'L$, or $S=S'R$, for
some $S'\in\{L,R\}^*$. Suppose
$S=S'L$. Then $P_R(S)=S'$ and $r(S'L)=r(S')-2^{-|S'L|}$ by (\ref{E:rdef}).
Thus $r(P_R(S))=r(S)+2^{-|S|}$. Since $S=S'L$, part~(c) gives $P_L(S)=P_L(S')$, 
and induction gives $r(P_L(S'))=r(S')-2^{-|S'|}$. Hence
\[
  r(P_L(S))=r(P_L(S'))=r(S')-2^{-|S'|}=r(S')-2^{-|S|}-2^{-|S|}=r(S)-2^{-|S|}.
\]
Similar arguments may be used to handle the case when $S=S'R$.

(f)~Formula~(\ref{E:ParStr}) needs interpretation when $m=0$.
In this case, $S=\eps$, $k_0=0$, $X=R^{-1}$, and $Y=L^{-1}$. This agrees
with $P_L(\eps)=R^{-1}$ and $P_R(\eps)=L^{-1}$. Suppose now that $m>0$.
The last symbol of $S$ is $R^{k_m}$ if $m$ is even, and $L^{k_m}$ if $m$ is odd,
where $k_m\geq1$. Formula~(\ref{E:ParStr}) now follows by canceling, as
$P_L(S)=SR^{-1}$ and $P_R(S)=SL^{-1}$.
\end{proof}

Counting the strings in $\{L,R\}^*$ in Fig.~\ref{F:Srev} from top down
and then left-to-right gives the tree in Fig.~\ref{F:rN}b.
Recursively define a bijective {\it position} function
$N\colon \{L,R\}^*\to\N$~by
\begin{equation}\label{E:N}
  N(\eps)=0,\quad N(SL)=2N(S)+1,\textup{ and}\quad N(SR)=2N(S)+2
  \qquad\textup{for $S\in\{L,R\}^*$}.
\end{equation}
Substituting $S=L^{-1}$ into $N(SL)=2N(S)+1$ gives $N(L^{-1})=-\frac12$,
and substituting $S=R^{-1}$ into $N(SR)=2N(S)+2$ gives $N(R^{-1})=-1$.

As a consequence of Theorem~\ref{T:main}(d), each $\eps\neq S\in\{L,R\}^*$
has a {\it close parent}, denoted $P_C(S)$, and a {\it distant parent},
denoted $P_D(S)$. Set $n:=N(S)$. Then $P_C(S)$ equals $P_L(S)$ if $n$ is even,
and $P_R(S)$ if $n$ is odd.
Similarly, $P_D(S)$ equals $P_L(S)$ if $n$ is odd, and $P_R(S)$ if $n$ is even.
Table~\ref{t:NPD} suggests that
$N(P_C(S))=\big\lfloor\frac{N(S)-1}2\big\rfloor$ holds.
This is easily proved. However, a formula for the numbers $N(P_D(S))$
is more mysterious.
\begin{table}[!ht]
  \begin{center}
    \begin{tabular}{|c|c|c|c|c|c|c|c|c|c|c|c|c|c|c|c|c|c|c|c|c|c|c|c|} \hline
    \st{$n\kern-1pt=\kern-1ptN(S)$}&\st{1}&\st{2}&\st{3}&\st{4}&\st{5}&\st{6}&\st{7}&\st{8}&\st{9}&\st{10}&\st{11}&\st{12}&\st{13}&\st{14}&\st{15}&\st{16}&\st{17}&\st{18}&\st{19}&\st{20}&\st{21}&\st{22}\\ \hline
    \st{\kern-3pt$N(P_C(S))$\kern-3pt}&\st{0}&\st{0}&\st{1}&\st{1}&\st{2}&\st{2}&\st{3}&\st{3}&\st{4}&\st{4}&\st{5}&\st{5}&\st{6}&\st{6}&\st{7}&\st{7}&\st{8}&\st{8}&\st{9}&\st{9}&\st{10}&\st{10}\\ \hline
    \st{\kern-3pt$N(P_D(S))$\kern-3pt}&\st{$-1$}&\st{$-\frac12$}&\st{$-1$}&\st{0}&\st{0}&\st{$-\frac12$}&\st{$-1$}&\st{1}&\st{1}&\st{0}&\st{0}
  &\st{2}&\st{2}&\st{$-\frac12$}&\st{$-1$}&\st{3}&\st{3}&\st{1}&\st{1}&\st{4}&\st{4}&\st{0}\\ \hline
    \end{tabular}
  \end{center}
  \vskip1mm
  \caption{Position numbers of close and distant parents}\label{t:NPD}
\end{table}
The integer-valued function $n\mapsto 2N(P_D(N^{-1}(n)))+1$ does not appear
(at the time of writing) in the On-Line Encyclopedia of Integer
Sequences, see~\url{http://oeis.org}. A formula for $N(P_D(S))$ can be computed from
(\ref{E:1},b) below.

\begin{theorem}\label{T:r}
Set $S:=S(k_0,\dots,k_{m-1},k_m)$. The functions $P_C$, $P_D$, $N$, and $r$
can be computed (nonrecursively) by the formulas
%$r$ recursively defined by \textup{(\ref{E:N})}
%and \textup{(\ref{E:rdef})} can be computed by the formulas
% satisfies $r(S)=\frac{2B(S)+1}{2^{|S|}}$ and
\begin{subequations}\label{E:0}
\begin{align}
  P_C(S)&=S(k_0,\dots,k_{m-1},k_m-1)\quad\textup{and}\quad
  P_D(S)=S(k_0,\dots,k_{m-2},k_{m-1}-1);\label{E:1}\\
  N(S)&=\begin{cases}
      2^{k_0+\cdots+k_m+1}-2^{k_1+\cdots+k_m}+2^{k_2+\cdots+k_m}-\cdots+2^{k_m}-2
        &\textup{if $m$ is even;}\\
      2^{k_0+\cdots+k_m+1}-2^{k_1+\cdots+k_m}+2^{k_2+\cdots+k_m}-\cdots-2^{k_m}-1
        &\textup{if $m$ is odd;}\\
      \end{cases}\label{E:2}\\
  r(S)&=2\left(1-2^{-k_0}+2^{-k_0-k_1}-\cdots+
  (-1)^m2^{-k_0-k_1-\cdots-k_{m-1}}+(-1)^{m+1}2^{-|S|-1}\right)  \label{E:3}\\
      &=2\left(\sum_{i=0}^{m+1}(-1)^i2^{-\eps_j}\right)
        +(-1)^{m+2}2^{-|S|}\quad\textup{where $\eps_j=\sum_{j=0}^{i-1} k_j$.}
  \label{E:4}
\end{align}
\end{subequations}
\end{theorem}

\begin{proof}
Equation (\ref{E:1}) follows easily from the rules for postmultiplying by
$R^{-1}$ or $L^{-1}$.
To verify formulas (\ref{E:2}) and (\ref{E:3}) it suffices to prove that
they satisfy their respective recurrence relations. This is somewhat
easier, paradoxically, than guessing the formulas in the first place.
We begin by showing that (\ref{E:2}) satisfies the recurrence~(\ref{E:N}).
First, $N(R^{k_0})=2^{k_0+1}-2$ holds by (\ref{E:2}). Setting $k_0=0$ shows
$N(\eps)=0$ which agrees with~(\ref{E:N}). Second, we must show $N(SL)=2N(S)+1$.
When $m$ is even, $S$ ends in $R^{k_m}$ and so
$SL=S(k_0,\dots,k_{m-1},k_m,1)$. It follows from the second line of (\ref{E:2})
that $N(SL)=2N(S)+1$ holds. If $m$ is odd, then $S$ ends in $L^{k_m}$ and
so $SL$ equals $S(k_0,\dots,k_{m-1},k_m+1)$. Again the second line of
(\ref{E:2}) implies that $N(SL)=2N(S)+1$ holds. Similar reasoning
involving the first line of (\ref{E:2}) shows that
$N(SR)=2N(S)+2$ holds, independent of the parity of $m$.
Hence (\ref{E:2}) is the (unique) solution to (\ref{E:N}).

The last term of  formula (\ref{E:3}) equals the last two terms of (\ref{E:4})
because
\[
  2(-1)^{m+1}2^{-|S|-1}=2(-1)^{m+1}2^{-|S|}+(-1)^{m+2}2^{-|S|}.
\]
Hence (\ref{E:3}) equals (\ref{E:4}).
We now prove that the solution to the recurrence relation (\ref{E:rdef}) is
given by the formula (\ref{E:3}). The base case $r(\eps)=1$ accords
with formula (\ref{E:3}). Suppose that $S=S(k_0,\dots,k_m)$ where $m$ is even.
Then $SR=S(k_0,\dots,k_{m-1},k_m+1)$.
%By (\ref{E:3}), the last two terms of $\frac{r(SR)}{2}$ are
%\[
%  (-1)^{m+1}2^{-|S|}+(-1)^{m+2}2^{-|SR|-1}=(-1)^{m+1}2^{-|S|-2}(2^2-1)
%  =3(-1)^{m+1}2^{-|S|-2}.
%\]
Comparing the expressions for $\frac{r(SR)}{2}$ and $\frac{r(S)}{2}$ given
by (\ref{E:3}) yields
\begin{subequations}\label{E:r0}
\begin{align}
  \frac{r(SR)}{2}&=\frac{r(S)}{2}-(-1)^{m+1}2^{-|S|-1}+(-1)^{m+1}2^{-|SR|-1}\label{E:r1}\\
  &=\frac{r(S)}{2}+(-1)^{m+2}2^{-|S|-2}(2-1)
  =\frac{r(S)}{2}+2^{-|SR|-1}.
\end{align}
\end{subequations}
If $m$ is even, then $SL=S(k_0,\dots,k_m,1)$.
By (\ref{E:3}), the last two terms of $\frac{r(SL)}{2}$ are
\[
  (-1)^{m+1}2^{-|S|}+(-1)^{m+2}2^{-|SL|-1}=(-1)^{m+2}2^{-|S|-2}(-2^2+1)
  =-3\cdot2^{-|S|-2}.
\]
Comparing the expressions for $\frac{r(SL)}{2}$ and $\frac{r(S)}{2}$ given
by (\ref{E:3}) yields
\begin{subequations}\label{E:r3}
\begin{align}
  \frac{r(SL)}{2}&=\frac{r(S)}{2}-(-1)^{m+1}2^{-|S|-1}-3\cdot2^{-|S|-2}\label{E:r4}\\
  &=\frac{r(S)}{2}+2^{-|S|-2}(2-3)
  =\frac{r(S)}{2}-2^{-|SL|-1}.
\end{align}
\end{subequations}
Equations (\ref{E:r0}) and (\ref{E:r3}) accord with the recurrence relation
(\ref{E:rdef}). The proof when $m$ is odd is similar. Hence the solution
to the the recurrence relation (\ref{E:rdef}) is (\ref{E:3}), as desired.
\end{proof}

In Fig.~\ref{F:MN} we compare the length function $N\colon\{L,R\}^*\to\N$
with another length function $M\colon\{L,R\}^*\to\N$ defined by $M(\eps)=0$ and
$M(S(k_0,k_1,\dots,k_m))=m$ if $k_m\geq1$.
\begin{figure}[!ht]
  \begin{center}
  \psset{xunit=4mm}\psset{yunit=7mm}%\psset{xunit=5mm}\psset{yunit=9mm}
  \begin{pspicture}(0,0)(16.5,3)
    \thicklines
    \rput(1,0){\st{$1$}}\rput(3,0){\st{$2$}}\rput(5,0){\st{$3$}}
    \rput(7,0){\st{$2$}}\rput(9,0){\st{$1$}}\rput(11,0){\st{$2$}}
    \rput(13,0){\st{$1$}}\rput(15,0){\st{$0$}}
    \rput(2,1){\st{$1$}}\rput(6,1){\st{$2$}}\rput(10,1){\st{$1$}}
    \rput(14,1){\st{$0$}}
    \rput(4,2){\st{$1$}}\rput(12,2){\st{$0$}}
    \rput(8,3){\st{$0$}}
    \psline(1.3,0.3)(1.7,0.7)\psline(5.3,0.3)(5.7,0.7)
    \psline(9.3,0.3)(9.7,0.7)\psline(13.3,0.3)(13.7,0.7)
    \psline(2.3,0.7)(2.7,0.3)\psline(6.3,0.7)(6.7,0.3)
    \psline(10.3,0.7)(10.7,0.3)\psline(14.3,0.7)(14.7,0.3)
    \psline(2.3,1.2)(3.7,1.85)\psline(10.3,1.2)(11.7,1.85)
    \psline(5.7,1.2)(4.3,1.85)\psline(13.7,1.2)(12.3,1.85)
    \psline(4.3,2.07)(7.7,2.92)\psline(8.3,2.92)(11.7,2.15)
  \end{pspicture}
  \hskip15mm
  \begin{pspicture}(0,0)(16.5,3)
    \thicklines
    \rput(1,0){\st{$7$}}\rput(3,0){\st{$8$}}\rput(5,0){\st{$9$}}
    \rput(7,0){\st{$10$}}\rput(9,0){\st{$11$}}\rput(11,0){\st{$12$}}
    \rput(13,0){\st{$13$}}\rput(15,0){\st{$14$}}
    \rput(2,1){\st{$3$}}\rput(6,1){\st{$4$}}\rput(10,1){\st{$5$}}
    \rput(14,1){\st{$6$}}
    \rput(4,2){\st{$1$}}\rput(12,2){\st{$2$}}
    \rput(8,3){\st{$0$}}
    \psline(1.3,0.3)(1.7,0.7)\psline(5.3,0.3)(5.7,0.7)
    \psline(9.3,0.3)(9.7,0.7)\psline(13.3,0.3)(13.7,0.7)
    \psline(2.3,0.7)(2.7,0.3)\psline(6.3,0.7)(6.7,0.3)
    \psline(10.3,0.7)(10.7,0.3)\psline(14.3,0.7)(14.7,0.3)
    \psline(2.3,1.2)(3.7,1.85)\psline(10.3,1.2)(11.7,1.85)
    \psline(5.7,1.2)(4.3,1.85)\psline(13.7,1.2)(12.3,1.85)
    \psline(4.3,2.07)(7.7,2.92)\psline(8.3,2.92)(11.7,2.15)
  \end{pspicture}
  \end{center}
  \caption{Values for the length functions (a) $M$, and (b) $N$.}
    \label{F:MN}
\end{figure}
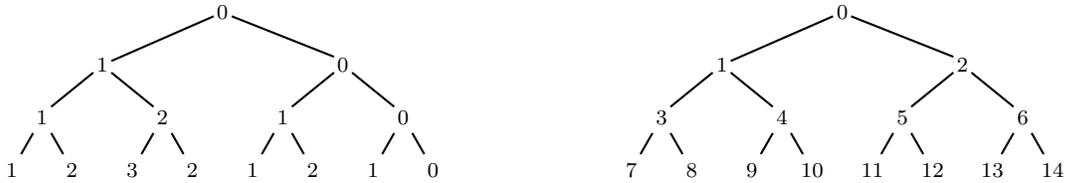

\begin{corollary}\label{C}
\textup{(a)} With the above definition, $M(S)\equiv N(S)\pmod 2$
for all $S\in\{L,R\}^*$.\n
\textup{(b)} Let $S=R^{k_0}L^{k_1}R^{k_2}\cdots$ and
$S'=R^{k'_0}L^{k'_1}R^{k'_2}\cdots$
be finite strings. Then $r(S)<r(S')$ if and only if $k_0<k'_0$, or
$k_0=k'_0$ and $k_1>k'_1$, or $(k_0,k_1)=(k'_0,k'_1)$ and $k_2<k'_2$, or
$(k_0,k_1,k_2)=(k'_0,k'_1,k'_2)$ and $k_3>k'_3$, \dots, using an
`alternating lexicographic' ordering.
\end{corollary}

\begin{proof} (a)~Certainly $M(S)\equiv N(S)\pmod 2$ holds when $S=\eps$,
as $M(\eps)=N(\eps)=0$. Assume $S:=S(k_0,\dots,k_m)$ and
$M(S)\equiv N(S)\pmod 2$ holds. If $m=M(S)$ is even, then
$S$ ends in $R^{k_m}$, so $SL=S(k_0,\dots,k_m,1)$ and
$M(SL)=M(S)+1$ is odd. Hence, by induction, and Eq.~\ref{E:N}
\[
  M(SL)=M(S)+1\equiv 2M(S)+1\equiv 2N(S)+1=N(SL)\pmod 2.
\]
Similarly, if $m$ is even, then
$SR=S(k_0,\dots,k_{m-1},k_m+1)$ and $M(SR)=M(S)$.
Hence
\[
  M(SR)=M(S)\equiv 2M(S)+2\equiv 2N(S)+2=N(SR)\pmod 2.
\]
If $M(S)=m$ is odd, then $S$ ends in $L^{k_m}$ and so $M(SL)=M(S)$
and $M(SR)=M(S)+1$ hold. Hence $M(SL)\equiv N(SL)\pmod 2$ and
$M(SR)\equiv N(SR)\pmod 2$ both hold.

(b)~The formula (\ref{E:3}) for $r(S)$ implies that Eq.~(\ref{E:LRorder})
holds. Hence moving left decreases the $r$-value, and moving right increase
the $r$-value. A left move is not counteracted by any number of right moves;
nor is a right move counteracted by any number of left moves. This
implies that the $r$-values are ordered via 
the stated alternating lexicographic ordering.
\end{proof}

\section{Continued fractions and the Stern-Brocot tree}\label{S:app}

In this section, we shall consider the tree $\TC$ in Fig.~\ref{F:TC}
whose vertices are continued fractions. Parents of vertices in this
tree are
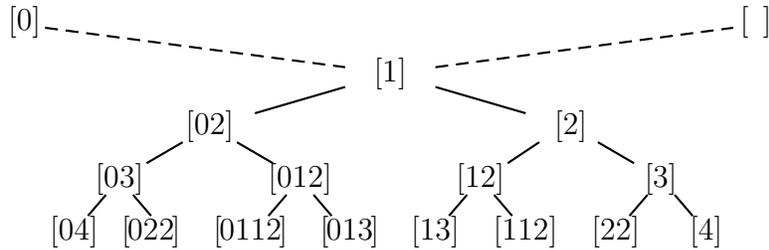
\begin{figure}[!ht]
  \begin{center}
  \psset{xunit=6mm}\psset{yunit=7mm}%\psset{xunit=5mm}\psset{yunit=9mm}
  \begin{pspicture}(-0.5,0)(16,4)
    \thicklines
    \rput(1,0){$[04]$}\rput(2.7,0){$[\kern-1pt022\kern-1pt]$}
    \rput(4.9,0){$[\kern-1pt0112\kern-1pt]$}
    \rput(7.1,0){$[\kern-1pt013\kern-1pt]$}\rput(9,0){$[13]$}
    \rput(11,0){$[112]$}
    \rput(13,0){$[22]$}\rput(15,0){$[4]$}
    \rput(2,1){$[03]$}\rput(6,1){$[012]$}\rput(10,1){$[12]$}
    \rput(14,1){$[3]$}
    \rput(4,2){$[02]$}\rput(12,2){$[2]$}
    \rput(8,3){$[1]$}\rput(-0.1,4){$[0]$}\rput(16.1,4){$[\;\,]$}
    \psline(1.3,0.3)(1.7,0.7)\psline(5.3,0.3)(5.7,0.7)
    \psline(9.3,0.3)(9.7,0.7)\psline(13.3,0.3)(13.7,0.7)
    \psline(2.3,0.7)(2.7,0.3)\psline(6.3,0.7)(6.7,0.3)
    \psline(10.3,0.7)(10.7,0.3)\psline(14.3,0.7)(14.7,0.3)
    \psline(2.6,1.3)(3.4,1.7)\psline(10.6,1.3)(11.3,1.7)
    \psline(5.4,1.3)(4.6,1.7)\psline(13.4,1.3)(12.6,1.7)
    \psline(5,2.25)(7,2.75)\psline(9,2.75)(11,2.25)
    \psline[linestyle=dashed](0.35,3.875)(7,3.125)
    \psline[linestyle=dashed](9,3.125)(15.6,3.875)
  \end{pspicture}
  \end{center}
  \caption{Binary tree $\TC$ of continued fractions (with commas omitted).}
    \label{F:TC}
\end{figure}
related to `best approximations.'
Recall that a continued fraction is an expression of the form
\[
  [q_0,q_1,\dots,q_m]=
q_0+\cfrac{1}{q_1+\cfrac{1}{\phantom{1}^{\ddots}+\cfrac{1}{q_m}}}.
%q_0+1/(\cdots+(q_{m-1}+1/q_m)).
%q_0+\cfrac1{\begin{matrix}\ddots& \\&+\cfrac{1}{q_{m-1}+\cfrac1{q_m}}\end{matrix}}.
\]
Continued fractions can be computed recursively via the recurrence
\begin{equation}\label{E:rCF}
  [q_0]=q_0\quad\textup{and}\quad
  [q_0,\dots,q_{m-1},q_m]=[q_0,\dots,q_{m-2},q_{m-1}+1/q_m]
  \quad\textup{for $m>0$}.
\end{equation}
The tree $\TC$ has root $[1]$, and its children rules are described
in Fig.~\ref{F:Ctree}
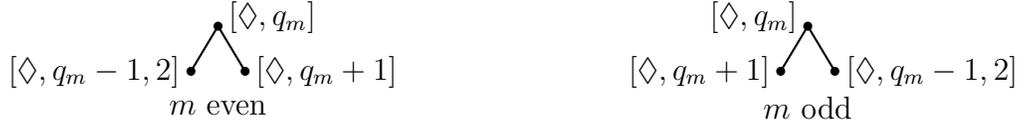
\begin{figure}[!ht]
  \begin{center}\psset{xunit=6mm,yunit=6mm}
    \begin{pspicture}(-3.5,-1.7)(3.5,0.4)
      \psline(-0.6,-1)(0,0)\psline(0.6,-1)(0,0)%\psline(0,0)(-1,1)
      \rput(1.2,0.2){$[\lozenge,q_m]$}
      \rput(-2.75,-1){$[\lozenge,q_m-1,2]$}\rput(2.4,-1){$[\lozenge,q_m+1]$}
      \rput(0,-1.8){\textup{$m$ even}}
      \pscircle*(0,0){0.06}\pscircle*(-0.6,-1){0.06}\pscircle*(0.6,-1){0.06}
    \end{pspicture}
    \hskip35mm
    \begin{pspicture}(-3.5,-1.7)(3.5,0.4)
      \psline(-0.6,-1)(0,0)\psline(0.6,-1)(0,0)
      \rput(-1.2,0.2){$[\lozenge,q_m]$}
      \rput(-2.4,-1){$[\lozenge,q_m+1]$}\rput(2.75,-1){$[\lozenge,q_m-1,2]$}
      \rput(0,-1.8){\textup{$m$ odd}}
      \pscircle*(0,0){0.06}\pscircle*(-0.6,-1){0.06}\pscircle*(0.6,-1){0.06}
    \end{pspicture}
    \caption{Children rules for the infinite complete binary tree $\TC$ in Fig.~\ref{F:TC}.}\label{F:Ctree}
  \end{center}
\end{figure}
where `$\lozenge$' is an abbreviation for `$q_0,\dots,q_{m-1}$.'
A simple induction proves that the continued fractions $[q_0,\dots,q_{m-1},q_m]$
generated each have $q_0\geq0$, $q_1,\dots,q_{m-1}\geq1$, and $q_m\geq2$ if $m>0$. (Incidentally, this
ensures that when these continued fractions are evaluated using
(\ref{E:rCF}) that no denominators of zero are encountered.)

%Another simple induction proves that the continued fraction in the position
%determined by the string $S(k_0,\dots,k_{m-1},k_m)$
%is $[k_0,\dots,k_{m-1},k_m-1]$. Thus by equation (\ref{E:1}), the close
%and distant parents of $[k_0,\dots,k_{m-1},k_m]$ are
%$[k_0,\dots,k_{m-1},k_m-1]$ and $[k_0,\dots,k_{m-1}]$ respectively.
%Suppose that these parents, when evaluated with the above recurrence, equal
%the fractions $\frac ab$ and $\frac cd$ respectively.
%If $[k_0,\dots,k_{m-1},k_m]$ evaluates to $\frac pq$, then $\frac ab$ is the
%smallest positive fraction which satisfies the equation $qa-pb=(-1)^m$,
%and $\frac cd$ is the largest? positive frac satisfying $qa-pb=(-1)^{m+1}$??

Let $\C$ be the set of vertices (i.e. continued fractions) of the infinite tree
$\TC$ in Fig.~\ref{F:TC}. When evaluated using (\ref{E:rCF}), a continued
fraction $[q_0,\dots,q_m]$ equals a positive rational~$\frac pq$. Positive
rationals have a natural ordering ($\frac{p_1}{q_1}<\frac{p_2}{q_2}$ if and only
if~$p_1q_2<p_2q_1$), so the continued fractions in $\C$ are naturally ordered.
We shall compare this ordering of $\C$ to the ordering of $\{L,R\}^*$
via the function $r$, see (\ref{E:rdef}). Towards this end, we define
a function $f\colon\C\to\{L,R\}^*$ by
$f([q_0,\dots,q_{m-1},q_m])=S(q_0,\dots,q_{m-1},q_m-1)$.

\begin{theorem}\label{T:app1}
Abbreviate $[q_0,\dots,q_{m-1},q_m]\in\C$ by $[\lozenge,q_m]$ where
\textup{`$\lozenge$'} means \textup{`$q_0,\dots,q_{m-1}$'.}
\begin{itemize}
  \item[(a)] The $LR$-location of $[\lozenge,q_m]$ in $\TC$ is given by the
  string $f([\lozenge,q_m])=S(\lozenge,q_m-1)$.
  \item[(b)] If $q_0,\dots,q_{m-1}$ is fixed, then $[\lozenge,q_m]$ is an
  increasing function of $q_m$ if $m$ is even, and a decreasing function
  of $q_m$ if $m$ is odd.
  \item[(c)] The function $f\colon\C\to\{L,R\}^*$ is a bijection preserving
  level, children, and order.
  \item[(d)] The parents of $[\lozenge,q_m]$ are the two closest smaller-level
  approximations
  to $[\lozenge,q_m]$. That is, of the $2^{q_0+\cdots+q_m}-1$ continued fractions
  $[q'_0,\dots,q'_n]$ with $\sum_{i=1}^nq'_i<\sum_{i=1}^m q_i$, the two
  closest to $[\lozenge,q_m]$ are the parent continued fractions
  of $[\lozenge,q_m]$.
\end{itemize}
\end{theorem}

\begin{proof}
(a)~Define the {\it length} of $[q_0,\dots,q_m]$ to be $q_0+\cdots+q_m-1$.
Our proof uses induction on the length of $[q_0,\dots,q_m]$.
The $LR$-location of $[q_0]$ in $\TC$ is $R^{q_0-1}$, and $f([q_0])=S(q_0-1)$.
In particular, the base case of length~0 (when $q_0\kern-1pt=\kern-1pt1$) holds.
Suppose now that $q_0+\cdots+q_m\kern-1pt>\kern-1pt 1$.
The length of the children of $[\lozenge,q_m]$ in
Fig.~\ref{F:Ctree} is one more than the length of $[\lozenge,q_m]$.
If $m$ is even, then induction gives
\begin{align*}
  &f([\lozenge,q_m-1,2])=f([\lozenge,q_m])L=S(\lozenge,q_m-1)L
  =S(\lozenge,q_m-1,1),\quad\textup{and}\\
  &f([\lozenge,q_m+1])=f([\lozenge,q_m])R=S(\lozenge,q_m-1)R=S(\lozenge,q_m).
\end{align*}
If $m$ is odd, then induction gives
\begin{align*}
  &f([\lozenge,q_m+1])=f([\lozenge,q_m])L=S(\lozenge,q_m-1)L=S(\lozenge,q_m),
  \quad\textup{and}\\
  &f([\lozenge,q_m-1,2])=f([\lozenge,q_m])R=S(\lozenge,q_m-1)R
  =S(\lozenge,q_m-1,1).
\end{align*}
Thus the function $f$ gives the $LR$-location of each continued fraction,
as desired.

(b)~We use induction on~$m$. Certainly $[q_0]=q_0$ is an increasing function of
$q_0$, and $[q_0,q_1]=q_0+\frac1{q_1}$ is a decreasing function of $q_1$.
Suppose now that $m\geq2$. If $m$ is even, then
$[q_0,\dots,q_m]=[q_0,\dots,q_{m-2},q_{m-1}+\frac1{q_m}]$ decreases (by induction)
precisely when $q_{m-1}+\frac1{q_m}$ increases. Thus $[\lozenge,q_m]$
increases when $q_m$ increases. If $m$ is odd, then
$[q_0,\dots,q_m]=[q_0,\dots,q_{m-2},q_{m-1}+\frac1{q_m}]$ increases
(by induction) precisely when $q_{m-1}+\frac1{q_m}$ increases. Thus
$[\lozenge,q_m]$ decreases when $q_m$ increases. This completes the induction.

(c)~It is clear that $f$ is surjective since
$f([k_0,\dots,k_{m-1},k_m+1])=S(k_0,\dots,k_{m-1},k_m)$ is a typical $LR$-string
in $\{L,R\}^*$. It is also clear that $f$ is injective, and hence $f$ is
bijective.
The level of $S(q_0,\dots,q_{m-1},q_m-1)$ is $q_0+\cdots+q_m-1$.
A simple induction shows that the level of $[q_0,\dots,q_m]$ in $\TC$ is
also $q_0+\cdots+q_m-1$. (This is true for the root $[1]$ of $\TC$.
If $[q_0,\dots,q_m]$ has level $q_0+\cdots+q_m-1$, then by Fig.~\ref{F:Ctree}
its children have level $q_0+\cdots+q_m$.) Part~(a) and the children rules
(Fig.~\ref{F:Ctree}) show that $f$ preserves
children, i.e. $f(C_L(v))=C_L(f(v))$ and $f(C_R(v))=C_R(f(v))$ for all
$v\in\C$.

It remains to prove that $f$ preserves order. This is true if $f^{-1}$
preserves order. This, in turn, amounts to proving that
the alternating lexicographic ordering of strings in Cor.~\ref{C}(b)
is the same as ordering of continued fractions (i.e. of the rational numbers).
Suppose that $S:=S(k_0,\dots,k_m-1)$ and $S':=S(k'_0,\dots,k'_n-1)$ where
$r(S)<r(S')$. By Cor.~\ref{C}(b) there exists an $i$ for which
$k_0=k'_0$, \dots, $k_{i-1}=k'_{i-1}$,
and $k_i<k'_i$ when $i$ is even, and $k_i>k'_i$ when $i$ is odd.
We shall prove the rational $v:=[k_0,\dots,k_m]$ is less
than $v':=[k'_0,\dots,k'_n]$.

Before proving the base case when $i=0$, we digress to prove
$k_0\leq[k_0,\dots,k_m]<k_0+1$.
This is true when $m=0$ as $[k_0]=k_0$. We prove that
$k_0<[k_0,\dots,k_m]<k_0+1$ holds for $m\geq1$. The proof of this
stronger statement uses induction on $m$. It is true
when $m=1$ as $k_0<[k_0,k_1]=k_0+\frac1{k_1}\leq k_0+\frac12$ because,
in our context, $k_m\geq2$ for $m\geq1$. Suppose the stronger statement
is true for $m-1$ where $m>1$. Then, by induction,
$k_1<[k_1,\dots,k_m]<k_1+1$ holds. Since $k_1\geq1$, taking inverses shows
\[
  \frac1{k_1+1}<[0,k_1,\dots,k_m]<\frac1{k_1},
\]
and adding $k_0$ implies
\[
  k_0<k_0+\frac1{k_1+1}<[k_0,k_1,\dots,k_m]<k_0+\frac1{k_1}\leq k_0+1.
\]
Hence $k_0<[k_0,k_1,\dots,k_m]<k_0+1$ holds for $m\geq1$, establishing the
digression.

%Return now to our induction on $i$. Consider the base case $i=0$.
Return now to the base case $i=0$ of our induction.
Certainly $k_0<k'_0$ implies $v<v'$ when $m$ or $n$ is zero. If $m$ and
$n$ are both positive, then $k_0<k'_0$ implies
\[
  k_0\leq v<k_0+1\leq k'_0\leq v'<k'_0+1.
\]
In either case, $k_0<k'_0$ implies $v<v'$.
Suppose now $i\geq1$ is even and $k_0=k'_0$, \dots, $k_{i-1}=k'_{i-1}$, and
$k_i<k'_i$. Abbreviate `$k_0,\dots,k_{i-1}$' by $\lozenge$.
Then Eq.~(\ref{E:rCF}) implies
\[
  v=[\lozenge,k_i,\dots,k_m]=[\lozenge,[k_i,\dots,k_m]]\quad\textup{and}\quad
  v'=[\lozenge,k'_i,\dots,k'_n]=[\lozenge,[k'_i,\dots,k'_n]].
\]
Hence the base case says $k_i<k'_i$ implies $[k_i,\dots,k_m]<[k'_i,\dots,k'_n]$.
Therefore part~(b) implies
$[\lozenge,[k_i,\dots,k_m]]<[\lozenge,[k'_i,\dots,k'_n]]$
and $v<v'$, as desired. The case when $i$ is odd is proved similarly.
This completes the proof that $S<S'$ implies $v<v'$.

(d)~This follows from the definition of parent fractions and part~(c).
\end{proof}

The continued fractions on the vertices of the tree
in Fig.~\ref{F:TC} give rise to
a tree of rational numbers. This is the well-known
Stern-Brocot tree~\cite[p.\,117]{GKP}, which (remarkably)
lists every positive rational number
(in reduced form) precisely once. To each string $S\in\{L,R\}^*$
there corresponds a {\it reverse string} $\overline{S}$ defined by
$\overline{\eps}=\eps$, $\overline{SL}=L\overline{S}$, and
$\overline{SR}=R\overline{S}$. Reversing
(or swapping $S\leftrightarrow\overline{S}$) vertices in a tree
gives another tree, called the {\it reverse tree}. The reverse tree of the
Stern-Brocot tree is another well-known tree called the Calkin-Wilf
tree~\cite{CW}. (In Fig.~\ref{F:SC}, $\frac23\leftrightarrow\frac32$ because
$LR\leftrightarrow RL$.)
The reader may better understand the connection between
parent and children vertices by studying the Stern-Brocot and
Calkin-Wilf trees, see~\cite{G11}.
%These connections are described in greater detail in a forthcoming book by the author.

\begin{figure}[!ht]
  \begin{center}
  \psset{xunit=4mm}\psset{yunit=7mm}%\psset{xunit=5mm}\psset{yunit=9mm}
  \begin{pspicture}(0,0)(16,4)
    \thicklines
    \rput(1,0){$\frac14$}\rput(3,0){$\frac25$}\rput(5,0){$\frac35$}
    \rput(7,0){$\frac34$}\rput(9,0){$\frac43$}\rput(11,0){$\frac53$}
    \rput(13,0){$\frac52$}\rput(15,0){$\frac41$}
    \rput(2,1){$\frac13$}\rput(6,1){$\frac23$}\rput(10,1){$\frac32$}
    \rput(14,1){$\frac31$}
    \rput(4,2){$\frac12$}\rput(12,2){$\frac21$}
    \rput(8,3){$\frac11$}\rput(0,4){$\frac01$}\rput(16,4){$\frac10$}
    \psline(1.3,0.3)(1.7,0.7)\psline(5.3,0.3)(5.7,0.7)
    \psline(9.3,0.3)(9.7,0.7)\psline(13.3,0.3)(13.7,0.7)
    \psline(2.3,0.7)(2.7,0.3)\psline(6.3,0.7)(6.7,0.3)
    \psline(10.3,0.7)(10.7,0.3)\psline(14.3,0.7)(14.7,0.3)
    \psline(2.3,1.15)(3.7,1.85)\psline(10.3,1.15)(11.7,1.85)
    \psline(5.7,1.15)(4.3,1.85)\psline(13.7,1.15)(12.3,1.85)
    \psline(4.3,2.07)(7.7,2.92)\psline(8.3,2.92)(11.7,2.15)
    \psline[linestyle=dashed](0.3,3.96)(7.7,3.04)
    \psline[linestyle=dashed](8.3,3.04)(15.7,3.96)
  \end{pspicture}
  \hskip20mm
  \psset{xunit=4mm}\psset{yunit=7mm}%\psset{xunit=5mm}\psset{yunit=9mm}
  \begin{pspicture}(0,0)(16,4)
    \thicklines
    \rput(1,0){$\frac14$}\rput(3,0){$\frac43$}\rput(5,0){$\frac35$}
    \rput(7,0){$\frac52$}\rput(9,0){$\frac25$}\rput(11,0){$\frac53$}
    \rput(13,0){$\frac34$}\rput(15,0){$\frac41$}
    \rput(2,1){$\frac13$}\rput(6,1){$\frac32$}\rput(10,1){$\frac23$}
    \rput(14,1){$\frac31$}
    \rput(4,2){$\frac12$}\rput(12,2){$\frac21$}
    \rput(8,3){$\frac11$}\rput(0,4){$\frac01$}\rput(16,4){$\frac10$}
    \psline(1.3,0.3)(1.7,0.7)\psline(5.3,0.3)(5.7,0.7)
    \psline(9.3,0.3)(9.7,0.7)\psline(13.3,0.3)(13.7,0.7)
    \psline(2.3,0.7)(2.7,0.3)\psline(6.3,0.7)(6.7,0.3)
    \psline(10.3,0.7)(10.7,0.3)\psline(14.3,0.7)(14.7,0.3)
    \psline(2.3,1.15)(3.7,1.85)\psline(10.3,1.15)(11.7,1.85)
    \psline(5.7,1.15)(4.3,1.85)\psline(13.7,1.15)(12.3,1.85)
    \psline(4.3,2.07)(7.7,2.92)\psline(8.3,2.92)(11.7,2.15)
    \psline[linestyle=dashed](0.3,3.96)(7.7,3.04)
    \psline[linestyle=dashed](8.3,3.04)(15.7,3.96)
  \end{pspicture}
  \end{center}
  \caption{(a) Stern-Brocot tree; and the reversed 
    (b) Calkin-Wilf tree.}\label{F:SC}
\end{figure}
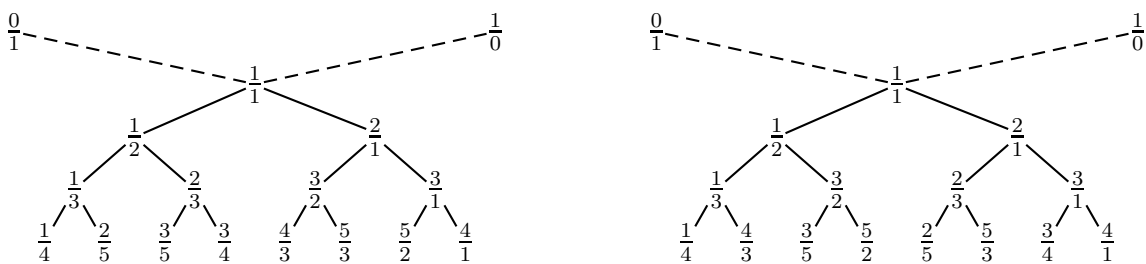
There are further applications of parent vertices to a complete binary tree
associated with the Cantor\footnote{The original discoverer appears to be
H.J.S. Smith~\cite[p.\,147]{S}, see~\cite{H} for details.} 
set, however, exploring these goes beyond the scope of this note.

\vskip1mm
\noindent
{\bf Acknowledgement:} The author is grateful to the referee for his/her suggestions.

\end{document}